\documentclass{amsart}

\newtheorem{theorem}{Theorem}[section]
\newtheorem{lemma}[theorem]{Lemma}

\theoremstyle{definition}

\theoremstyle{remark}
\newtheorem{remark}[theorem]{Remark}

\numberwithin{equation}{section}

\begin{document}

\title{Local Product Structure of Expansive Flows with Shadowing Property}

%    Information for first author
\author{Xiao Wen}
%    Address of record for the research reported here
\address{School of Mathematical Sciences, Beihang University, Beijing, 100191, People's Republic of China}

\email{wenxiao@buaa.edu.cn}
%    \thanks will become a 1st page footnote.
\thanks{The author was supported by National Key R\&D Program of China (No. 2022YFA1005801).}

%    General info


\begin{abstract}
%% Text of abstract

We show that there exists a suspension flow defined over the full shift map which does not have local product structure. This gives a negative answer to an open problem posed by Gelfert-Kwietniak-Lima recently (\cite{GKL}).
\end{abstract}

\maketitle

\section{Introduction}
Let $X$ be a compact metric space with distance $\rho$ and $\varphi_t$ be a flow on $X$.  The flow is called {\it expansive} if for any $\varepsilon>0$, there is $\delta>0$  such that for any $x, y\in X$ and any homeomorphism $s: \mathbb{R}\to \mathbb{R}$, if $\rho(\varphi_{s(t)}(y), \varphi_t(x))<\delta$ for all $t\in\mathbb{R}$, then $\varphi_{s(t)}(y)\in\varphi_{[-\varepsilon, \varepsilon]}(\varphi_t(x))$ for all $t\in\mathbb{R}$ (see\cite{BW}). It is well known that an expansive flow has all its fixed points isolated. Given $\delta>0$, we call a sequence $\{(x_i, t_i)\}_{i=a}^b(-\infty\leq a<b\leq +\infty)$  of points $x_i\in X$ and $t_i\geq 1$ a {\it $\delta$-pseudo orbit} of flow $\varphi_t$ if $\rho(\varphi_{t_i}(x_i), x_{i+1})<\delta$ for every $i\in\mathbb{Z}$ with $a\leq i<b$. We say that a $\delta$-pseudo orbit $(x_i, t_i)_{i=-\infty}^\infty$ is $\varepsilon$-{\it traced} by the orbit of $y\in X$ if there exists an increasing homeomorphism $h: \mathbb{R}\to\mathbb{R}$ satisfying $h(0)=0$ and the following:
$$\rho(\varphi_{h(t)}(y), \varphi_{t-s_k}(x_k))<\varepsilon$$
for every $k\in\mathbb{Z}$ and $s_k\leq t<s_{k+1}$ where $s_k$ is defined by $s_0=0, s_k=t_0+\cdots+t_{k-1} (k>0)$ and $s_{k}=t_k+\cdots+t_{-1} (k<0)$. We say that $\varphi_t$ has the {\it shadowing property} if for every $\varepsilon>0$, there exists $\delta>0$ such that every $\delta$-pseudo orbit $(x_i, t_i)_{i=-\infty}^\infty$ can be $\varepsilon$-traced by an orbit of a point $y\in X$.

Given $x\in X$, the {\it strong stable set} of $x$ is defined by
$$W^s(x)=\{y\in X: \rho(\varphi_t(y), \varphi_t(x))\to 0 \text{ as } t\to+\infty \}.$$
Similarly, the strong unstable set of $x$ is given by
$$W^u(x)=\{y\in X: \rho(\varphi_t(y), \varphi_t(x))\to 0 \text{ as } t\to-\infty \}.$$
Given $\varepsilon>0$, denote by
$$W_\varepsilon^s(x)=\{y\in W^s(x): \rho(\varphi_t(y), \varphi_t(x))\leq \varepsilon, \text{ for all } t\geq 0\},$$
$$W_\varepsilon^u(x)=\{y\in W^u(x): \rho(\varphi_t(y), \varphi_t(x))\leq \varepsilon, \text{ for all } t\leq 0\}.$$
If for every sufficiently small $\varepsilon>0$ there is $\delta>0$ such that for every $x,y\in X$
with $\rho(x,y)\leq\delta$ there is a unique $\theta =\theta(x,y)\in\mathbb{R}$ with $|\theta|\leq\varepsilon$ such that the set
$W_\varepsilon^s(\varphi_\theta(x))\cap W_\varepsilon^u(y)$
contains exactly one point, which we denote by $[x,y]$, then we say the flow has {\it local product structure}.

It is well known that local product structure and the shadowing property are equivalent for expansive homeomorphisms (see \cite{Omb}). Based on this, Gelfert-Kwietniak-Lima \cite{GKL} recently proposed the following:
\bigskip

\noindent{\bf Open Problem}. Does the shadowing property imply local product structure for expansive fixed
point free flows?
\bigskip

The following theorem gives a negative answer to the above problem.
\begin{theorem}
There exists a fixed point free expansive flow with pseudo-orbit tracing property which does not possess local product structure.
\end{theorem}

We construct the corresponding example in the following way. First, we take the symbolic space $\Sigma$ and the left shift map $\sigma$ on $\Sigma$. Then we choose a roof function $\tau$ that fails to be H\"{o}lder continuous, from which we obtain the suspension flow $\varphi_t$ defined on the suspension space $X_{\tau}$ w.r.t the roof function $\tau$. Equipped with the Bowen-Walters metric on $X_{\tau}$, we prove that the flow $\varphi_t$ does not possess the local product structure.

Here the non-H\"{o}lder continuity for the roof function $\tau$ is sharp. As a comparison, we also present the following theorem.

\begin{theorem}
Let $\sigma:\Sigma\to\Sigma$ be the left shift map on the two-symbol symbolic space $\Sigma$, and $\tau_1:\Sigma\to\mathbb{R}^+$ be a H\"{o}lder continuous roof function. Let $X_{\tau_1}$ be the suspension space and $\psi_t$ the induced suspension flow with roof function $\tau_1$. Then $\psi_t$ admits the local product structure.
\end{theorem}

\section{Proof of Theorem 1.1}
In this section, we prove Theorem 1.1 by constructing a suspension flow over the full shift map which has no local product structure.
Let \(\Sigma=\{0, 1\}^{\mathbb{Z}}\) and let \(\sigma\colon \Sigma\to\Sigma\) denote the left shift map. We define the metric on \(\Sigma\) by
\[
d(a, b)=\frac{1}{3}\sum_{i\in\mathbb{Z}}\frac{|a_i-b_i|}{2^{|i|}},
\]
for any two sequences \(a=(a_i)_{i\in\mathbb{Z}}, b=(b_i)_{i\in\mathbb{Z}}\in\Sigma\). Denote by
$$W_\varepsilon^s(a, \sigma)=\{b\in\Sigma: d(\sigma^n(b), \sigma^n(a))\leq \varepsilon, \text{ for all } n\geq 0\},$$
$$W_\varepsilon^u(a, \sigma)=\{b\in\Sigma: d(\sigma^n(b), \sigma^n(a))\leq \varepsilon, \text{ for all } n\leq 0\},$$
for given $a\in\Sigma$ and $\varepsilon>0$. It is easy to see that $d(\sigma^{\pm 1}(a), \sigma^{\pm 1}(b))\leq 2d(a,b)$ for all $a, b\in \Sigma$. Consequently, $\sigma^{-1}(W_\varepsilon^s(\sigma(a)), \sigma))\subset W^s_{2\varepsilon}(a, \sigma)$ and $\sigma(W_\varepsilon^u(\sigma^{-1}(a)), \sigma)\subset W^u_{2\varepsilon}(a, \sigma)$ for all $a\in\Sigma$.

Let \(\underline{0}=(0)_{i\in\mathbb{Z}}\) denote the zero sequence in \(\Sigma\). We define a positive continuous roof function \(\tau\colon\Sigma\to\mathbb{R}^+\) by
\[
\tau(a)=
\begin{cases}
\displaystyle 1+\frac{1}{-\ln d(a, \underline{0})+1}, & a\in\Sigma\setminus\{\underline{0}\},\\[6pt]
1, & a=\underline{0}.
\end{cases}
\]
Note here that $\tau$ is not H\"{o}lder continuous at $\underline{0}$.

Consider the suspension flow \(\varphi_t\) of \((\Sigma,\sigma)\) with roof function \(\tau\). The suspension space is defined as
\[
X_\tau=\big\{(a, t)\mid a\in \Sigma,\, 0\leq t\leq \tau(a)\big\}\big/\sim,
\]
where the identification is given by \((a, \tau(a))\sim (\sigma(a), 0)\). The flow \(\varphi_t\) acts on \(X_\tau\) by vertical unit-speed translation: for any \((a,s)\in X_\tau\) and \(t\geq 0\) with \(s+t\leq \tau(a)\), we have \(\varphi_t(a, s)=(a, s+t)\).

We now recall the Bowen-Walters metric on the suspension space \(X_\tau\). We first define a standard suspension space with constant roof function \(\tau_0 \equiv 1\). Let \(X_1\) be the suspension space of \((\Sigma,\sigma)\) associated with \(\tau_0\), i.e.,
\[
X_1=\Sigma\times [0,1]\big/ (x,1)\sim (\sigma(x),0).
\]
Denote by \(\phi_t\) the corresponding suspension flow on \(X_1\). For any two points \(x_1,x_2\in X_1\), consider a finite chain \(x_1=w_1,w_2,\dots,w_n=x_2\) in \(X_1\). We say this chain is {\it admissible } if for each \(1\le k\le n-1\), either \(w_k\) and \(w_{k+1}\) lie on the same vertical fiber \(\Sigma\times\{t\}\) for some \(t\in[0,1]\), or \(w_{k+1}\) belongs to the flow orbit of \(w_k\) under \(\phi_t\). For each \(k\), write \(w_k=(a_k,t_k)\) with \(a_k\in\Sigma\) and \(t_k\in[0,1]\). We define the elementary path length \(|w_k w_{k+1}|\) as follows:
\begin{enumerate}
\item If \(w_k,w_{k+1}\in\Sigma\times\{t\}\) for some common \(t\in[0,1]\), set
\[
|w_k w_{k+1}|=(1-t)d(a_k,a_{k+1})+t\,d(\sigma(a_k),\sigma(a_{k+1})).
\]
\item If no such common fiber exists, then \(w_{k+1}\) lies on the flow orbit of \(w_k\). In this case, define
\[
|w_k w_{k+1}|=\min\big\{|t|\in\mathbb{R} \,\big|\, \phi_t(w_k)=w_{k+1}\big\}.
\]
\end{enumerate}
The total length of the chain \(w_1,\dots,w_n\) is given by \(\sum\limits_{k=1}^{n-1}|w_k w_{k+1}|\). The metric \(\rho_0\) on \(X_1\) is then defined as the infimum of the lengths of all admissible chains connecting two points $x_1, x_2\in X_1$:
\[
\rho_0(x_1,x_2)=\inf\bigg\{\sum_{k=1}^{n-1}|w_k w_{k+1}|: \text{ all admissible paths } w_1=x_1,\dots,w_n=x_2\bigg\}.
\]
Using \(\rho_0\), we define the Bowen-Walters metric \(\rho\) on the general suspension space \(X_\tau\). For any two points \((a,t),(b,s)\in X_\tau\) with \(0\le t\leq\tau(a)\) and \(0\le s\leq\tau(b)\), set
\[
\rho\big((a,t),(b,s)\big)=\rho_0\bigg(\Big(a,\frac{t}{\tau(a)}\Big),\Big(b,\frac{s}{\tau(b)}\Big)\bigg).
\]

\begin{lemma}\label{lemma1.2}
For all \((a,t)\in X_\tau\) with \(0\leq t\leq \tau(a)\) and all \(b\in\Sigma\) satisfying \(\rho\big((a,t),(b,\tau(b)/2)\big)<1/4\), we have
\[
\rho\big((a,t),(b,\tau(b)/2)\big)\geq \min\big\{d(a,b),d(\sigma(a),\sigma(b))\big\},
\]
\[\rho\big((a,t),(b,\tau(b)/2)\big)\geq\big|\frac{t}{\tau(a)}-\frac{1}{2}\big|.\]
\end{lemma}

\begin{proof}
Let \((a,t)\) be a point in $X_\tau$ with \(0\leq t\leq\tau(a)\), and let \(b\in\Sigma\) satisfy \(\rho\big((a,t),(b,\tau(b)/2)\big)<1/4\). Choose an admissible chain $$(a, t/\tau(a))=w_1, w_2, \cdots, w_n=(b, 1/2)$$ in $X_1$ such that
$$\sum_{k=1}^{n-1}|w_k w_{k+1}|<\frac{1}{4}.$$
Write $w_k=(a_k, t_k)$ with $a_k\in\Sigma$ and $t_k\in[0,1]$. We first claim that $t_k\in[1/4, 3/4]$ for all $1\leq k\leq n$. Suppose for contradiction that there is some $1\leq k\leq n-1$ such that $t_k\notin[1/4, 3/4]$. Let $k_1=\max\{k: t_k\notin [1/4, 3/4]\}$.
We call the segment $w_kw_{k+1}$ vertical if there exist no $t\in[0,1]$ such that $w_k, w_{k+1}$ are both in $\Sigma\times\{t\}$. If a vertical segment $w_kw_{k+1}$ satisfies $k>k_1$ and $a_k\neq a_{k+1}$, then we can see that the orbit segment with end points $w_k$ and $w_{k+1}$ crosses the section $\Sigma\times\{0\}$ or $\Sigma\times\{1\}$, which forces $|w_kw_{k+1}|\geq 1/4$. This contradicts the bound on the total length. Therefore, for any vertical segment $w_kw_{k+1}$ with $k>k_1$, we have $a_k=a_{k+1}$ and $|w_kw_{k+1}|\geq |t_k-t_{k+1}|$ at the same time. By the definition of $k_1$, the segment $w_{k_1}w_{k_1+1}$ is also a vertical one. If $a_{k_1}=a_{k_1+1}$, then $|w_{k_1}w_{k_1+1}|\geq |t_{k_1}-t_{k_1+1}|$. It follows that
$$\sum_{k=1}^{n-1}|w_k w_{k+1}|\geq\sum_{k=k_1}^{n-1}|w_kw_{k+1}|\geq \sum_{k=k_1}^{n-1}|t_k-t_{k+1}|\geq |t_{k_1}-t_n|=|t_{k_1}-\frac{1}{2}|\geq\frac{1}{4},$$
a contradiction. Consequently, $a_{k_1}\neq a_{k_1+1}$. As argued above, the orbit segment with end points $w_k, w_{k+1}$ acrosses the section $\Sigma\times\{0\}$ or $\Sigma\times\{1\}$, which again yields  $|t_{k_1}-t_{k+1}|\geq 1/4$, another contradiction. We conclude that $t_k\in[1/4, 3/4]$ for all $1\leq k\leq n$. A similar argument shows that $a_k=a_{k+1}$ for any vertical segment $w_{k}w_{k+1}$.

A segment $w_kw_{k+1}$ is horizontal if it is not vertical. For any horizontal segment $w_kw_{k+1}$,
\[\begin{array}{lll}|w_kw_{k+1}| & = & (1-t_k)d(a_k, a_{k+1})+t_{k}d(\sigma(a_k), \sigma(a_{k+1}))\\ &\geq &\min\{d(a_k, a_{k+1}), d(\sigma(a_k), \sigma(a_{k+1}))\}.\end{array}\]
Note that we have $\min\{d(a_k, a_{k+1}), d(\sigma(a_k), \sigma(a_{k+1}))\}=0$ for a vertical segment $w_kw_{k+1}$ because of $a_k=a_{k+1}$. Hence we have
$$\begin{array}{lll}\sum\limits_{k=1}^{n-1}|w_kw_{k+1}|&\geq& \sum\limits_{k=1}^{n-1}\min\{d(a_k, a_{k+1}), d(\sigma(a_k), \sigma(a_{k+1}))\}\\ &\geq&\min\{d(a, b), d(\sigma(a), \sigma(b))\}.\end{array}$$
Similarly, we have $$\sum\limits_{k=1}^{n-1}|w_kw_{k+1}|\geq\sum_{k=1}^{n-1}|t_{k-1}-t_k|\geq |t_0-t_n|=\big|\frac{t}{\tau(a)}-\frac{1}{2}\big|.$$
Taking the infimum over all this kind of admissible chains yields $$\rho((a, t), (b, \tau(b)/2))\geq\min\{d(a, b), d(\sigma(a), \sigma(b))\},$$
\[\rho\big((a,t),(b,\tau(b)/2))\big)\geq\big|\frac{t}{\tau(a)}-\frac{1}{2}\big|.\]
This finishes the proof of the lemma.
\end{proof}

\begin{remark}\label{remark2.1}
We remark here that the estimations in the above claim hold for all suspension flow with a continuous roof function $\tau$.
\end{remark}
It is well known that the full shift \((\Sigma,\sigma)\) is expansive and possesses the shadowing property. By Theorem 6 in \cite{BW} and Theorem 2 in \cite{Tho}, the resulting suspension flow \(\varphi_t\) is a fixed-point-free expansive flow on \(X_\tau\) with the shadowing property.

To complete the proof of Theorem 1.1, it suffices to establish the following key lemma.

\begin{lemma}\label{lemma1.1}
There is $\varepsilon_0>0$ such that for any $0<\varepsilon<\varepsilon_0$, $$W_\varepsilon^s((\underline{0}, 0))=\{(\underline{0}, 0)\}, \ \ \ W_\varepsilon^u((\underline{0}, 0))=\{(\underline{0}, 0)\}.$$
\end{lemma}

\begin{proof}
For any sequences \(a=(a_i), b=(b_i)\in\Sigma\), if \(d(\sigma^n(a), \sigma^n(b))< 1/3\) holds for all integers \(n\geq 0\), then \(a_i = b_i\) for all \(i\geq 0\), and  \(d(\sigma^n(a), \sigma^n(b))=2^{-n}d(a, b)\) hold for all \(n\geq 1\). We fix an arbitrary constant \(0<\varepsilon<1/12\) for subsequent analysis.

Suppose \(x=(a, \theta_0)\in W_\varepsilon^s((\underline{0}, 0))\) for some \(a\in\Sigma\) and \(\theta_0\in[0,\tau(a))\). By the definition of the suspension flow, we have
\[
\varphi_{\tau(a)/2-\theta_0}(x)=(a, \tau(a)/2), \ \ \ \ \varphi_{\tau(a)-\theta_0}(x)=(\sigma(a),0),
\]
\[
\varphi_{\tau(\sigma(a))/2+\tau(a)-\theta_0}(x)=(\sigma(a), \tau(\sigma(a))/2),
\]
and inductively,
\[
\varphi_{\tau(\sigma^n(a))/2+\tau(\sigma^{n-1}(a))+\tau(\sigma^{n-2}(a))+\cdots+\tau(a)-\theta_0}(x)=(\sigma^n(a),\tau(\sigma^n(a))/2)
\]
for all $n\geq 1$. Define the sequence $\{\tau_n\}$ as
\[
\tau_n=\frac{\tau(\sigma^n(a))}{2}+\sum_{k=0}^{n-1}\tau(\sigma^k(a))-\theta_0,\quad n\geq 1,
\]
and let \(\tau_n'=\tau_n\bmod 1\). Since \(\sigma(\underline{0})=\underline{0}\) and \(\tau(\underline{0})=1\), we have \(\varphi_{\tau_n}((\underline{0},0))=(\underline{0},\tau_n')\).

Since \(x\in W_\varepsilon^s((\underline{0},0))\), one has
$$\rho((\sigma^n(a), \tau(\sigma^n(a))/2), (\underline{0}, \tau_n'))=\rho(\varphi_{\tau_n}(x), \varphi_{\tau_n}((\underline{0}, 0))\leq \varepsilon$$
for all \(n\geq 1\). Then by Lemma \ref{lemma1.2}, we conclude
\(
\min \{d(\sigma^n(a), \underline{0}), d(\sigma^{n+1}(a), \underline{0})\}\leq \varepsilon
\)
for all $n\geq 1$. Consequently, \(d(\sigma^n(a),\underline{0})\leq2\varepsilon\) for all \(n\geq 1\), which implies that $\sigma(a)\in W_{2\varepsilon}^s(\underline{0}, \sigma)$ and
\[
d(\sigma^n(a),\underline{0})=2^{-n}d(a,\underline{0})
\]
for all $n\geq 0$.

We now consider two cases. If \(a=\underline{0}\), then by \(\tau(\underline{0})=1\), we know \(\tau_n=\frac{1}{2}+n-\theta_0\) and
\[
\rho\big(\varphi_{\tau_n}(x),\varphi_{\tau_n}(\underline{0},0)\big)
=\rho\big((\underline{0},\frac{1}{2}),(\underline{0},\frac{1}{2}-\theta_0)\big).
\]
The limit \(\rho(\varphi_{\tau_n}(x),\varphi_{\tau_n}(\underline{0},0))\to 0\) as \(n\to\infty\) forces \(\theta_0=0\).
Next suppose \(a\neq\underline{0}\). For this case,
\[
\begin{aligned}
\tau_n
&=\frac{\tau(\sigma^n(a))}{2}+\sum_{k=0}^{n-1}\tau(\sigma^k(a))-\theta_0\\
&=\frac{1}{2}+\frac{1/2}{-\ln\big(2^{-n}d(a,\underline{0})\big)+1}+\sum_{k=0}^{n-1}\left[1+\frac{1}{-\ln\big(2^{-k}d(a,\underline{0})\big)+1}\right]-\theta_0\\
&=n+\sum_{k=0}^{n-1}\frac{1}{k\ln 2-\ln d(a,\underline{0})+1}+\frac{1}{2}+\frac{1}{2(n\ln 2-\ln d(a, \underline{0}) +1)}-\theta_0.
\end{aligned}
\]
One can check that \(\tau_n-n\to+\infty\) and \(\tau_{n+1}-\tau_n-1\to 0\) as \(n\to\infty\). Therefore, the sequence \(\{\tau_n'\}\) is dense in \([0,1]\), which implies that the sequence \(\{\varphi_{\tau_n}(\underline{0},0)\}=\{(\underline{0},\tau_n')\}\) is dense in the orbit segment \(\{(\underline{0},t)\mid 0\leq t\leq 1\}\).
This contradicts the limit \(\rho(\varphi_{\tau_n}(x),\varphi_{\tau_n}(\underline{0},0))=\rho(\sigma^n(a), \tau(\sigma^n(a)/2), (\underline{0}, \tau_n'))\to 0\) as \(n\to\infty\). Hence we obtain \(W_\varepsilon^s((\underline{0},0))=\{(\underline{0},0)\}\). The identical argument applied to the backward flow gives \(W_\varepsilon^u((\underline{0},0))=\{(\underline{0},0)\}\).
\end{proof}

We now finalize the proof of Theorem 1.1. Fix a constant \(0<\varepsilon<\tfrac{1}{12}\). For arbitrary \(\delta>0\), we select a point \(a\in\Sigma\) satisfying \(\rho((a,0),(\underline{0},0))<\delta\) and \(a\notin W_{4\varepsilon}^s(\underline{0}, \sigma)\cup W_{4\varepsilon}^u(\underline{0}, \sigma)\). Suppose that there exists a real number \(\tau_0\) with \(|\tau_0|<\varepsilon\) such that
\[
W_\varepsilon^s\big(\varphi_{\tau_0}(a,0)\big)\cap W_\varepsilon^u\big((\underline{0},0)\big)\neq\emptyset.
\]
By Lemma \ref{lemma1.1}, we obtain \((\underline{0},0)\in W_\varepsilon^s(\varphi_{\tau_0}(a,0))\), which in turn implies \(\varphi_{\tau_0}(a,0)=(a, \tau_0)\in W_\varepsilon^s((\underline{0}, 0))\). Repeating the computational arguments from the proof of Lemma \ref{lemma1.1} yields \(\sigma(a)\in W_{2\varepsilon}^s(\underline{0}, \sigma)\), and consequently \(a\in W_{4\varepsilon}^s(\underline{0}, \sigma)\). This contradicts the choice of \(a\). Therefore, one has
\(
W_\varepsilon^s\big(\varphi_{\tau_0}(a,0)\big)\cap W_\varepsilon^u\big((\underline{0},0)\big)=\emptyset
\) for every $|\tau_0|<\varepsilon$.
Similarly we have \(
W_\varepsilon^u\big(\varphi_{\tau_0}(a,0)\big)\cap W_\varepsilon^s\big((\underline{0},0)\big)=\emptyset,
\) for every $|\tau_0|<\varepsilon$.
Thus the flow \(\varphi_t\) does not admit the local product structure.

\section{Proof of Theorem 1.2}

In the following, we assume that $\tau_1:\Sigma\to \mathbb{R}^+$ is a H\"{o}lder continuous function, that is, there is $L\geq 1$ and $\alpha\in (0, 1]$ such that $$|\tau_1(x)-\tau_1(y)|\leq Ld(x, y)^\alpha$$ for every $x, y\in\Sigma$. Let $\psi_t$ be the suspension flow on the suspension space $X_{\tau_1}$ w.r.t. the roof function $\tau_1$. For simplicity of notation, we still use $\rho$ to denote the Bowen-Walters distance on $X_{\tau_1}$.

\begin{lemma}\label{lemma3.1}
There exists $\varepsilon_0>0$ such that for any $0<\varepsilon<\varepsilon_0$, if $(a, t)\in W_\varepsilon^s((b, \tau_1(b)/2))$ $(0\leq t\leq\tau_1(a))$, then $a\in W^s_{2\varepsilon}(b, \sigma)$ and
$$t=\frac{\tau_1(b)}{2}+\sum\limits_{n=0}^{\infty}\big[\tau_1(\sigma^n(a))-\tau_1(\sigma^n(b))\big].$$
\end{lemma}
\begin{proof}
Let $K=\max\limits_{x\in\Sigma}\tau_1(x)$. Choose $\varepsilon_0\in(0, 1/6)$ such that $$L(2\varepsilon_0)^\alpha+K\varepsilon_0<\frac{1}{2}\min\limits_{x\in\Sigma}\tau_1(x).$$ Fix any $0<\varepsilon<\varepsilon_0$. Suppose $(a, t)\in W_\varepsilon^s((b, \tau_1(b)/2))$ for some $b\in\Sigma$ and $0\leq t\leq\tau_1(a)$. Let $t_1=\tau_1 (b)/2+\tau_1(\sigma(b))/2$, $t_2=t_1+\tau_1(\sigma(b))/2+\tau_1(\sigma^2(b))/2$ and $$t_n=t_{n-1}+\frac{\tau_1(\sigma^{n-1}(b))}{2}+\frac{\tau_1(\sigma^n(b))}{2}$$ be defined inductively. Then we have $$\psi_{t_n}((b, \frac{\tau_1(b)}{2}))=(\sigma^n(b), \frac{\tau_1(\sigma^n(b))}{2}),$$
and
$$\psi_{t_n}((a, t))=(\sigma^n(a), t+t_n-\tau_1(a)-\tau_1(\sigma(a))-\cdots-\tau_1(\sigma^{n-1}(a))).$$
Denote by $t_n'=t+t_n-\tau_1(a)-\tau_1(\sigma(a))-\cdots-\tau_1(\sigma^{n-1}(a))$.

By the fact $\rho((a, t), (b,\tau_1(b)/2))<\varepsilon$, we have
$\min\{d(a, b), d(\sigma(a), \sigma(b))\}<\varepsilon$ and $\big|\frac{t}{\tau_1(a)}-\frac{1}{2}\big|<\varepsilon$ (see Remark \ref{remark2.1}). Therefore, we have
$d(a,b)<2\varepsilon, d(\sigma(a), \sigma(b))<2\varepsilon$ and
$$\big|t-\frac{\tau_1(a)}{2}\big|<\varepsilon\tau_1(a).$$
By definition, we can see that
$$\begin{array}{ll}\big|t_1'-\frac{\tau_1(\sigma(a))}{2}\big|& = \big|t+\frac{\tau_1(b)}{2}+\frac{\tau_1(\sigma(b))}{2}-\tau_1(a)-\frac{\tau_1(\sigma(a))}{2}\big|\\
 & \leq \big|t-\frac{\tau_1(a)}{2}\big|+\frac{1}{2}\big|\tau_1(b)-\tau_1(a)\big|+\frac{1}{2}\big|\tau_1(\sigma(b))-\tau_1(\sigma(a))\big|\\
 & \leq \varepsilon\tau_1(a)+\frac{1}{2}L(d(a,b))^\alpha+\frac{1}{2}L(d(\sigma(a),\sigma(b)))^\alpha\\
 & \leq K\varepsilon+L(2\varepsilon)^\alpha<\frac{1}{2}\tau_1(\sigma(a)).
\end{array}$$
This proves that $0<t_1'<\tau_1(\sigma(a))$. By the fact $$\rho((\sigma(a), t_1'), (\sigma(b), \tau_1(\sigma(b))))=\rho(\psi_{t_1}(a, t), \psi_{t_1}(b, \tau_1(b)/2))<\varepsilon,$$
we have $\min \{d(\sigma(a), \sigma(b)), d(\sigma^2(a), \sigma^2(b))\}<\varepsilon$ and $\big|\frac{t_1'}{\tau_1(\sigma(a))}-\frac{1}{2}\big|<\varepsilon$.
Inductively, we can prove that
$$\min\{d(\sigma^n(a), \sigma^n(b)), d(\sigma^{n+1}(a), \sigma^{n+1}(b))\}<\varepsilon,$$
$$\big|\frac{t_n'}{\tau_1(\sigma^n(a))}-\frac{1}{2}\big|<\varepsilon$$
for all $n\geq 1$. This yields $d(\sigma^n(a), \sigma^n(b))\leq 2\varepsilon$ for all $n\geq 0$. Hence $a\in W_{2\varepsilon}^s(b, \sigma)$.
By the assumption that $\rho((\sigma^n(a), t_n'), (\sigma^n(b), \tau_1(\sigma^n(b))/2))\to 0$, we have
$$\big|t_n'-\frac{\tau_1(\sigma^n(b))}{2}\big|\to 0.$$
Therefore,
$$t-\frac{\tau_1(b)}{2}+\sum_{k=0}^{n-1}\big[\tau_1(\sigma^k(b))-\tau_1(\sigma^k(a))\big]\to 0$$
as $n\to\infty$. This yields
$$t=\frac{\tau_1(b)}{2}+\sum\limits_{n=0}^{\infty}\big[\tau_1(\sigma^n(a))-\tau_1(\sigma^n(b))\big],$$
and finishes the proof of the lemma.
\end{proof}

Similarly, we can choose $\varepsilon_0>0$ such that for any $0<\varepsilon<\varepsilon_0$, if $(a, t)\in W_\varepsilon^u((b, \tau_1(b)/2))$ $(0\leq t\leq\tau_1(a))$, then $a\in W^u_{2\varepsilon}(b, \sigma)$ and
$$t=\frac{\tau_1(b)}{2}+\sum\limits_{n=1}^{\infty}\big[\tau_1(\sigma^{-n}(b))-\tau_1(\sigma^{-n}(a))\big].$$

\begin{lemma}\label{lemma3.2}
For any $\varepsilon>0$, there exists $\epsilon>0$ such that for any $a\in W_{\epsilon}^s(b, \sigma)$ and any $0\leq t\leq\tau_1(b)$, we have
$$(a, s)\in W_{\varepsilon}^s((b, t)),$$
where $s=t+\sum\limits_{n=0}^{\infty}\big[\tau_1(\sigma^n(a))-\tau_1(\sigma^n(b))\big]$.
\end{lemma}
\begin{proof}
Let $K=\max\limits_{x\in\Sigma}\tau_1(x)$. Given $\varepsilon>0$. We can fix $\varepsilon_1>0$ such that for any $(a, s), (b, t)\in X_{\tau_1}$, $\rho((a, s), (b,t))<\varepsilon_1$ implies $\rho(\psi_t(a, s), \psi_t(b, t))<\varepsilon$ for every $-K\leq t\leq K$.

Let $(b,t)\in X_{\tau_1}$ be given with $0\leq t\leq \tau_1(b)$ and $a\in W_{1/3}^s(b, \sigma)$. Let $$s=t+\sum\limits_{n=0}^{\infty}\big[\tau_1(\sigma^n(a))-\tau_1(\sigma^n(b))\big].$$
Define $t_1=\tau_1(b)-t$, $t_2=\tau_1(\sigma(b))+\tau_1(b)-t$, $\cdots, t_n=\sum\limits_{k=0}^{n-1}\tau_1(\sigma^k(b))-t$ for every $n\geq 1$. Then we have $\psi_{t_n}((b, t))=(\sigma^n(b), 0)$ and
$$\begin{array}{lll}\psi_{t_n}((a, s)) &= &(\sigma^n(a), s+t_n-\tau_1(a)-\cdots-\tau_1(\sigma^{n-1}(a)))\\ & = & (\sigma^n(a), \sum\limits_{k=n}^{\infty}\big[\tau_1(\sigma^k(a))-\tau_1(\sigma^k(b))\big]).\end{array}$$
Denote by $t_n'=\sum\limits_{k=n}^{\infty}\big[\tau_1(\sigma^k(a)-\tau_1(\sigma^k(b)))\big].$ By the Bowen-Walters definition of the distance for the suspension flow, noting that $w_1=(\sigma^n(a), \frac{t_n'}{\tau_1(\sigma^n(a))}), w_2=(\sigma^n(a), 0), w_3=(\sigma^n(b), 0)$ is an admissible chain connecting $(\sigma^n(a), \frac{t_n'}{\tau_1(\sigma^n(a))})$ and $(\sigma^n(b), 0)$, we have
$$\rho(\psi_{t_n}(a, s), \psi_{t_n}(b,t))=\rho((\sigma^n(a), t_n'), (\sigma^n(b), 0))\leq d(\sigma^n(a), \sigma^n(b))+\frac{|t_n'|}{\tau_1(\sigma^n(a))}$$
for all $n\geq 1$. Since $\tau_1$ is H\"{o}lder continuous, we have
$$\begin{array}{llll}|t_n'| & = & \left|\sum\limits_{k=n}^{\infty}\big[\tau_1(\sigma^k(a))-\tau_1(\sigma^k(b))\big]\right|  \leq \sum\limits_{n=1}^{\infty}\big|\tau_1(\sigma^k(a))-\tau_1(\sigma^k(b))\big|\\
& \leq & \sum\limits_{n=1}^{\infty} L(d(\sigma^n(a), \sigma^n(b)))^\alpha \leq L\sum\limits_{n=1}^{\infty} (2^{-n}d(a,b))^\alpha = \frac{L}{1-2^{-\alpha}}[d(a, b)]^\alpha.& \end{array}$$
Now we take $\epsilon>0$ such that
$$\epsilon+\frac{L}{(1-2^{-\alpha})\min\limits_{x\in\Sigma}\tau_1(x)}\epsilon^\alpha<\varepsilon_1.$$
Given any $a\in W_\epsilon^s(b, \sigma)$,  we can see that
$$\begin{array}{lll}\rho(\psi_{t_n}(a, s), \psi_{t_n}(b,t)) & \leq & d(\sigma^n(a), \sigma^n(b))+\frac{|t_n'|}{\tau_1(\sigma^n(a))}\\ & \leq & \epsilon+\frac{L}{(1-2^{-\alpha})\min\limits_{x\in\Sigma}\tau_1(x)}\epsilon^\alpha\\ & < &\varepsilon_1\end{array}$$
for all $n\geq 1$. And we have $\rho(\psi_{t_n}(a, s), \psi_{t_n}(b,t))\to 0$ by the fact $d(\sigma^n(a), \sigma^n(b))\to 0, |t_n'|\to 0$. By the choice of $\varepsilon_1$ we can check that $\rho(\psi_{t}(a, s), \psi_{t}(b,t))\leq\varepsilon$ for all $t\geq 0$. This proves that $(a, s)\in W_{\varepsilon}^s(b, t)$.
\end{proof}

Similarly, we can prove that there exist $\epsilon>0$ such that $a\in W_\epsilon^u(b, \sigma)$ implies
$$(a, t+\sum\limits_{n=1}^{\infty}\big[\tau_1(\sigma^{-n}(b))-\tau_1(\sigma^{-n}(a))\big])\in W_{\varepsilon}^u((b, t))$$
for every $t\in[0, \tau_1(b)]$.

Now we complete the proof of Theorem 1.2. Let $K=\max\limits_{x\in\Sigma}\tau_1(x)$. Fix any $\varepsilon>0$, we can choose $\varepsilon_1>0$ such that $\rho(\psi_{t}(x),\psi_t(y))<\varepsilon$ for any $x, y\in X_{\tau_1}$ with $d(x,y)<\varepsilon_1$ and $t\in[-K, K]$. We can easily see that $\psi_t(W_{\varepsilon_1}^s(x))\subset W_{\varepsilon}^s(\psi_t(x))$ and $\psi_t(W_{\varepsilon_1}^u(x))\subset W_{\varepsilon}^u(\psi_t(x))$ hold for every $x\in X_{\tau_1}$ and $t\in[-K, K]$. We can choose $\epsilon>0$ small enough such that
$$\frac{K}{2}\epsilon+\frac{L}{2}\epsilon^\alpha+\frac{2L}{1-2^{-\alpha}}\epsilon^\alpha<\varepsilon.$$
By Lemma \ref{lemma3.2}, we can also choose $\epsilon>0$ small enough such that $a\in W_\epsilon^s(b, \sigma)$ implies that
$$(a, t+\sum\limits_{n=0}^{\infty}\big[\tau_1(\sigma^n(a))-\tau_1(\sigma^n(b))\big])\in W_{\varepsilon_1}^s((b,t))\ \  (0\leq t\leq\tau_1(b)), $$
 and $a\in W_{\epsilon}^u(b, \sigma)$ implies that
$$(a, t+\sum\limits_{n=1}^{\infty}\big[\tau_1(\sigma^{-n}(b))-\tau_1(\sigma^{-n}(a))\big])\in W_{\varepsilon_1}^u((b,t))\ \ (0\leq t\leq\tau_1(b).$$
Since $(\Sigma, \sigma)$ has local product structure, there is $0<\delta_0<\epsilon$ such that $d(a, b)<\delta_0$ implies that $W^s_{\epsilon}(a, \sigma)\cap W_{\epsilon}^u(b, \sigma)\neq\emptyset$. Now we can choose $\delta>0$ such that $\rho(x, y)<\delta$ implies $\rho(\psi_t(x), \psi_t(y))<\delta_0/2$ holds for every $t\in[-K, K]$.

Fix any two points $x, y\in X_{\tau_1}$ with $\rho(x,y)<\delta$. We can find $t_0\in[-K/2, K/2]$ such that $\psi_{t_0}(y)=(b, \tau_1(b)/2)$ for some $b\in\Sigma$. By the choice of $\delta$, we have $\rho(\psi_{t_0}(x),\psi_{t_0}(y))<\delta_0/2$. Denote by $\psi_{t_0}(x)=(a, t)$ where $a\in \Sigma$ and $0\leq t\leq\tau_1(a)$. Similar to Lemma \ref{lemma1.2}, we have $\min\{d(a, b), d(\sigma(a), \sigma(b))\}\leq\rho(\psi_{t_0}(x),\psi_{t_0}(y))<\delta_0/2$ and then $d(a, b)<\delta_0$. By the choice of $\delta_0$, there exists $c\in W_\epsilon^s(a, \sigma)\cap W_{\epsilon}^u(b, \sigma)$. By the choice of $\epsilon$, we have
$$(c, t+\sum\limits_{n=0}^{\infty}\big[\tau_1(\sigma^n(c))-\tau_1(\sigma^n(a))\big])\in W_{\varepsilon_1}^s((a, t)),$$
$$(c, \frac{\tau_1(b)}{2}+\sum\limits_{n=1}^{\infty}\big[\tau_1(\sigma^{-n}(b))-\tau_1(\sigma^{-n}(c))\big])\in W_{\varepsilon_1}^u(b, \frac{\tau_1(b)}{2}).$$
Let $$\theta=\frac{\tau_1(b)}{2}+\sum\limits_{n=1}^{\infty}\big[\tau_1(\sigma^{-n}(b))-\tau_1(\sigma^{-n}(c))\big]-(t+\sum\limits_{n=0}^{\infty}\big[\tau_1(\sigma^n(c))-\tau_1(\sigma^n(a))\big]).$$
We have $\psi_{\theta}(W_{\varepsilon_1}^s((a, t)))\cap W_{\varepsilon_1}^u((b, \tau_1(b)/2))\neq\emptyset$. Note that $\rho((a, t), (b, \tau_1(b)/2))<\delta_0/2$, we have $|(\tau_1(a))^{-1}t-1/2|<\delta_0/2$, then
$$\left|\frac{\tau_1(b)}{2}-t\right|\leq |t-\frac{\tau_1(a)}{2}|+\frac{1}{2}|\tau_1(a)-\tau_1(b)|\leq \frac{\delta_0}{2}\tau_1(a)+\frac{L}{2}d(a,b)^{\alpha}<\frac{K}{2}\delta_0+\frac{L}{2}\delta_0^\alpha.$$
Thus
$$\begin{array}{lll}|\theta| & \leq & \left|\frac{\tau_1(b)}{2}-t\right|+\sum\limits_{n=0}^{\infty}\big|\tau_1(\sigma^n(a))-\tau_1(\sigma^n(c))\big|+\sum\limits_{n=1}^{\infty}\big|\tau_1(\sigma^{-n}(b))-\tau_1(\sigma^{-n}(c))\big|\\
 & < &  \frac{K}{2}\delta_0+\frac{L}{2}\delta_0^\alpha+ \sum\limits_{n=0}^{\infty}Ld(\sigma^n(a), \sigma^n(c))^\alpha+ \sum\limits_{n=1}^{\infty}L d(\sigma^{-n}(b), \sigma^{-n}(c))^\alpha\\
 & \leq  &  \frac{K}{2}\delta_0+\frac{L}{2}\delta_0^\alpha+L\sum\limits_{n=0}^{\infty}[2^{-n}d(a, c)]^\alpha+L\sum\limits_{n=1}^{\infty}[2^{-n}d(b,c)]^\alpha\\
 & \leq  &\frac{K}{2}\delta_0+\frac{L}{2}\delta_0^\alpha+\frac{2L}{1-2^{-\alpha}}\epsilon^\alpha\leq \frac{K}{2}\epsilon+\frac{L}{2}\epsilon^\alpha+\frac{2L}{1-2^{-\alpha}}\epsilon^\alpha<\varepsilon.
 \end{array}$$
By the choice of $\varepsilon_1$, we know that $$W_\varepsilon^s(\psi_{\theta}(x))\cap W_{\varepsilon}^u(y)\supset\psi_{-t_0}(W_{\varepsilon_1}^s(\psi_\theta((a, t)))\cap W_{\varepsilon_1}^u((b, \tau_1(b)/2)))\neq \emptyset.$$

Now we suppose that we choose $\varepsilon>0$ small enough such that $\psi_{t}(W_\varepsilon^s(x))\subset W_{\varepsilon_0}^s(\psi_t(x))$ and $\psi_{t}(W_\varepsilon^u(x))\subset W_{\varepsilon_0}^u(\psi_t(x))$ for every $x\in X_{\tau_1}$ and $t\in[-K, K]$, where $\varepsilon_0$ is the constant given by Lemma \ref{lemma3.1}. We will show that if $W_\varepsilon^s(\psi_{\theta}(x))\cap W_{\varepsilon}^u(y)\neq \emptyset$ for some $|\theta|<\varepsilon$, then it contains exactly one point. Let $z\in W_\varepsilon^s(\psi_{\theta}(x))\cap W_{\varepsilon}^u(y)$, recall that we have taken $t_0\in[-K/2, K/2]$ such that $\psi_{t_0}(y)=(b, \tau_1(b)/2)$ and $\psi_{t_0}(x)=(a, t)$. By the choice of $\varepsilon$, we have $\psi_{t_0}(z)\in W_{\varepsilon}^u((b, \tau_1(b)/2))$.
By Lemma \ref{lemma3.1} we know that there is $c\in W_{2\varepsilon}^u(b, \sigma)$ such that
$$\psi_{t_0}(z)=(c, \frac{\tau_1(b)}{2}+\sum\limits_{n=1}^{\infty}\big[\tau_1(\sigma^{-n}(b))-\tau_1(\sigma^{-n}(c))\big]))=(c, s).$$
Note that $\psi_{t_0+\frac{\tau_1(a)}{2}-t}(x)=(a, \tau_1(a)/2)$, if $\varepsilon>0$ is chosen small, then $|t_0+\frac{\tau_1(a)}{2}-t|<K$ by the fact that $|t_0|\leq K/2$, $|\frac{\tau_1(a)}{2}-t|<\frac{\delta_0}{2}\tau_1(a)\leq \frac{\delta_0}{2}K$,  hence we have
$$\psi_{t_0+\frac{\tau_1(a)}{2}-t-\theta}(z)\in\psi_{t_0+\frac{\tau_1(a)}{2}-t-\theta}(W_{\varepsilon}^s(\psi_{\theta}(x)))\subset W_{\varepsilon_0}^s((a, \tau_1(a)/2)).$$
There is $c'\in W_{2\varepsilon}^s(a, \sigma)$ such that
$$\psi_{t_0+\frac{\tau_1(a)}{2}-t-\theta}(z)=(c', \frac{\tau_1(a)}{2}+\sum_{n=0}^{\infty}\big[\tau_1(\sigma^n(c'))-\tau_1(\sigma^n(a))\big])=(c', s').$$
Note that $s$ is close to $\tau_1(b)/2$, $s'$ is close to $\tau_1(a)/2$ and $(c, s)$ and $(c', s')$ is on the same orbit that
$$(c',s')=\psi_{\frac{\tau_1(a)}{2}-t-\theta}((c,s)).$$ Hence we have $c=c'$, then by the uniqueness of $W_{2\varepsilon}^s(a, \sigma)\cap W_{2\varepsilon}^u(b, \sigma)$ we know that $z$ is uniquely determined. This proves that $W_\varepsilon^s(\psi_{\theta}(x))\cap W_{\varepsilon}^u(y)$ contains exactly one point.

\end{document}